\documentclass{article}
\usepackage{amssymb}
\usepackage{amsmath}
\usepackage{amsthm}
\usepackage{caption}
\usepackage{subcaption}
\usepackage{xypic}
\usepackage{graphicx}
\usepackage{pgfplots}
\usepackage{hyperref}
\usepackage{kbordermatrix}
\usepackage{blkarray}
\usepackage{color} 
\usepackage{verbatim}   
\usepackage{bm}
\usepackage{MnSymbol}  
\usepackage[section]{placeins}  

\theoremstyle{theorem}
\newtheorem{theorem}{Theorem}
\newtheorem{prop}[theorem]{Proposition}
\newtheorem{lemma}[theorem]{Lemma}

\newtheorem*{conj}{Conjecture}
\newtheorem*{main}{The Main Theorem}

\newcommand{\Rb}{\mathbb{R}}
\newcommand{\Nb}{\mathbb{N}}
\newcommand{\Zb}{\mathbb{Z}}
\newcommand{\Qb}{\mathbb{Q}}
\newcommand{\Cb}{\mathbb{C}}

\newcommand{\PP}{\mathcal{P}}
\newcommand{\LL}{\mathcal{L}}
\newcommand{\QQ}{\mathcal{Q}}
\newcommand{\im}{\operatorname{Im}}
\newcommand{\T}{\top}

\author{Benjamin Ellis, David A. Nash,\\ Jonathan Needleman, and Michael Raney}
\date{}
\title{Non-magic Hypergraphs}
\begin{document}



\maketitle
\begin{abstract}
This article studies a generalization of magic squares to $k$-uniform hypergraphs.  In traditional magic squares the entries come from the natural numbers.  A magic labeling of the vertices in a graph or hypergraph has since been generalized to allow for labels coming from any abelian group.   We demonstrate an algorithm for determining whether a given hypergraph has a magic labeling over some abelian group.  A slight adjustment of this algorithm also allows one to determine whether a given hypergraph can be magically labeled over $\Zb$.  As a demonstration, we use these algorithms to determine the number of magic $n_3$-configurations for $n=7, \dots, 14$.
\end{abstract}

The notion of a magic square has existed for thousands of years, with surviving written examples dating back to at least $650$ BC.  One of the oldest written examples is the so-called ``Lo Shu'' magic square -- a $3\times 3$ magic square from Chinese legends.  According to the legend, the Lo Shu square was observed as a pattern on the shell of a tortoise by Emperor Yu sometime between 2200 and 2100 BC.  As a simple idea that is often fiendishly difficult to implement, it is no wonder that magic squares have been studied and marveled at throughout history.  The idea is to fill in numbers into a square so that the sum along each row, column and diagonal are all equal to the same number -- often called the magic constant.  For instance the square in Figure~\ref{fig:square} is a representation of the ``Lo Shu'' magic square from the legends.  Here the magic constant is 15.

\begin{figure}[!htbp]
\captionsetup[subfigure]{labelformat=simple}
\centering
\begin{tikzpicture}
    \draw[thick] (0,0) -- (2,0);
    \draw[thick] (0,.667) -- (2,.667);
    \draw[thick] (0,1.333) -- (2,1.333);
    \draw[thick] (2,2) -- (0,2);
    \draw[thick] (0,0) -- (0,2);
    \draw[thick] (.667,0) -- (.667,2);
    \draw[thick] (1.333,0) -- (1.333,2);
    \draw[thick] (2,0) -- (2,2);
    \draw (.333,.333) node {$8$};
    \draw (1,.333) node {$1$};
    \draw (1.667,.333) node {$6$};
    \draw (.333,1) node {$3$};
    \draw (1,1) node {$5$};
    \draw (1.667,1) node {$7$};
    \draw (.333,1.667) node {$4$};
    \draw (1,1.667) node {$9$};
    \draw (1.667,1.667) node {$2$};
\end{tikzpicture}
\caption{The ``Lo Shu'' magic square.}
\label{fig:square}
\end{figure}

There have been numerous generalizations of magic squares to other shapes.  Ely introduced the idea of magic designs \cite{Ely}.  A design is an incidence structure consisting of ``points'' and ``lines,'' with each line being a subset of points.   A \emph{magic design} is then an injective function from the points to the natural numbers where the sum along any line is constant.  While Ely focused primarily on triangles and hexagons, other designs have since been studied (see e.g.\ \cite{Trotter}, \cite[Ch.17]{Gardner}, \cite{Hendricks}, and \cite{MN}).  A particulary nice family of designs comes from (combinatorial) configurations.  These are designs where every line has the same number of points, every point has the same number of lines through it, no two lines can intersect more than once, and no two points can appear on more than one line together.  For example, Trenkler \cite{Trenkler:3} studied magic stars which are configurations with two lines through every point.  More recently, Raney \cite{Raney} studied magic trilateral free $n_3$-configurations where three lines pass through every point, and each line contains three points.  In addition, Nash and Needleman \cite{NN2} studied magic finite projective planes which are configurations which contain a quadrilateral and further require all lines to intersect, and all pairs of points to be connected by a line.

On a different, but related front, Sedl$\acute{\text{a}}\check{\text{c}}$ek was the first to suggest studying magic labelings of graphs in \cite{Sedlacek}.  He suggested labeling the edges of a graph with real numbers and then defined an edge-labeling to be magic if the sum of the labels on all edges incident to a given vertex was a constant independent of the choice of vertex.  Stewart \cite{Stewart} studied some special cases of these labelings and called an edge-labeling \emph{supermagic} if the labels were consecutive integers starting from 1.  Note that the classical $n\times n$ magic square (ignoring diagonals) corresponds to a supermagic labeling of the complete bipartite graph $K_{n,n}$.  Many, many papers (see e.g.\ \cite{Doob2}, \cite{Jeurissen}, \cite{SGL}, \cite{Ivanco}, and \cite{LLSW} as examples) have since been devoted to determining whether and when various families of graphs can admit a supermagic edge-labeling (although many authors use the term ``magic" rather than supermagic).  For a more complete treatment of this history (and some of what follows) we suggest the ongoing survey by Gallian \cite{Gallian}.

In a departure from the above notion, Kotzig and Rosa \cite{KR1,KR2} defined a \emph{magic valuation} of a graph to be a labeling of both the vertices and the edges (now referred to as a \emph{total labeling} following Wallis \cite{Wallis}) with the consecutive integers starting from 1 such that the sum of any edge together with its two endpoints is a constant.  This idea was revived in 1996 when Ringel and Llad$\acute{\text{o}}$ \cite{RL} redefined these as \emph{edge-magic} (total) labelings.  Edge-magic total labelings of graphs have also been widely studied (see e.g.\ \cite{ELNR} and \cite{WBMS} for some important early examples).

More recently, MacDougall, Miller, Slamin, and Wallis \cite{MMSW} moved back towards Sedl$\acute{\text{a}}\check{\text{c}}$ek's original idea by introducing the notion of a \emph{vertex-magic} total labeling.  Such a labeling still labels both the vertices and the edges using the consecutive integers starting with 1, but instead requires that the sum of the labels on all edges incident to a given vertex plus the label on the vertex itself be a constant.  These magic graphs have also been extensively studied, with \cite{GM1}, \cite{GM2}, and \cite{AM} as more recent examples of work in this area.

Trenkler \cite{Trenkler:2} was the first to generalize supermagic edge-labelings to include \emph{hypergraphs}.  A hypergraph is a generalization of a graph which allows for multiple vertices to appear on each (hyper)edge.  In Trenkler's work a labeling of the complete $k$-partite hypergraph is supermagic if the sum of the labels on all edges incident to a fixed choice of $(k-1)$ vertices is independent of the choice of vertices.  Even more recently,  Boonklurb, Narissayaporn, and Singhun \cite{BNS} generalized super edge-magic labelings to hypergraphs as well.  Their generalization more closely parallels the original notion by requiring the sum of the label on an edge together with the labels on all vertices appearing on that edge to be a constant.  To be supermagic, all edge labels must be greater than any vertex labels in addition to the labels being consecutive integers starting from 1.

Generalizations have also been made to the labels over time.  While historically, $n \times n$ magic squares have most often been labeled by the integers 1 through $n^2$ -- and supermagic labelings are an attempt to hold on to this requirement -- there is no need for us to require such specificity.  To consider a function or labeling magic one simply needs a way to ``add'' the outputs or labels from each vertex and/or edge. Since the points in a graph, hypergraph, incidence structure, design, etc.\ do not appear in any fixed order, it is also necessary that the ``addition'' be commutative.  Thus, Doob \cite{Doob1} generalized magic edge-labelings to allow for labels coming from any abelian group.  He further added labels to the vertices and used those as restrictions on the sums of edge labels incident to each vertex while simultaneously relaxing the requirement that distinct edges have distinct labels.  In the special case where the sums are constant independent of vertex choice, these labelings have since been termed \emph{group-magic} edge-labelings (note, these are also sometimes referred to as $A$-magic edge-labelings, where $A$ denotes the (additive) abelian group used).  Some more recent work in this area can be seen in \cite{LSW,SLS,LL,Salehi,SL,JJ1,JJ2}.

One thing that nearly all magic shapes have in common (regardless of the type of magic labeling considered) is that they require each edge or line to contain the same number of points.  From the hypergraph perspective, such structures are called \emph{$k$-uniform} hypergraphs.  In what follows we consider the general case of $k$-uniform hypergraphs with vertices labeled by elements from an (additive) abelian group.  Note, we do not require our labels to be non-zero elements of the group (as in the $A$-magic situation) and we \emph{do} require that distinct vertices have distinct labels.

Observe, that we can incorporate the notion of an edge-magic total labeling (generalized to $k$-uniform hypergraphs) by imagining the label on each edge/line to be replaced by a label on an extra vertex created on that line that is not incident to any other line.  Our version of a magic labeling also serves to cover the cases of magic edge-labelings and vertex-magic total labelings by considering the dual (hyper)graphs where the sum along all edges incident to a fixed vertex becomes the sum of all vertices contained within a given edge.  Thus, we choose to follow \cite{NN2} in describing magic vertex-labelings of $k$-uniform hypergraphs over abelian groups.

In what follows, we demonstrate algorithms for determining both whether a given hypergraph is magic over some abelian group and also whether it is magic specifically over $\Zb$.  The methods draw upon many results from number theory, and is an application of the local-global principle.

Much of the motivation for our discoveries stemmed from work specifically on $n_3$-configurations.  Thus, in the final section, we specialize to $n_3$-configurations and describe some results gained from the application of our algorithms to this type of hypergraph.

\subsection*{Labeling Hypergraphs}
While traditional graphs are collections of vertices with edges connecting them, they have been generalized to \emph{hypergraphs} by allowing edges to connect (or contain) more than two vertices.  A hypergraph can, therefore, be represented as a pair $\Gamma=(\PP,\LL)$ where $\PP$ is the set of \emph{points} and $\LL$ is a collection of subsets of $\PP$ that we call the set of \emph{lines}.  A hypergraph is called \emph{$k$-uniform} if each line has exactly $k$ points on it.

A hypergraph is called a \emph{configuration} if, in addition:
\begin{itemize}
\item[(a)] Each pair of points $p_i \neq p_j \in \PP$ is contained in at most one line.
\item[(b)] Each pair of lines $L_i \neq L_j \in \LL$ intersects at at most one point.
\item[(c)] Each line is incident to the same number of points.
\item[(d)] Each point is incident to the same number of lines.
\end{itemize}
A configuration with $n$ points, each incident to $k$ lines, and $m$ lines, each incident to $l$ points is called an \emph{$(n_k, m_l)$-configuration}.  Note that such configurations can only occur if $nk=ml$.  A configuration is called \emph{symmetric} in the special case when $n=m$, in which case it follows that $k=l$, and these are often referred to as \emph{$n_k$-configurations}.  Figure \ref{fig:hypergraph} is an example of a $4$-uniform hypergraph which is not a configuration since the points $p_1$ and $p_2$ are simultaneously coincident to more than one line.

To study labelings on a  hypergraph $\Gamma=(\PP,\LL)$ we construct the $\Zb$-modules
\[\Zb\PP=\left\{\left.\sum_{p\in \PP}a_p p~\right|~a_p\in \Zb\right\} \text{~and~} \Zb\LL=\left\{\left.\sum_{L\in \LL}b_L L~\right|~b_L\in \Zb\right\} \]
of the formal $\Zb$-linear combinations of points and lines respectively.

 The hypergraph $\Gamma$ determines a relationship between $\Zb\PP$ and $\Zb\LL$ via a unique $\Zb$-map $A:\PP\rightarrow \Zb\LL$ known as the \emph{incidence map}.  For $p\in \PP$, let
 \[Ap=\sum_{{L\in \LL} \atop {p \in L}}L\]
 and then extend $\Zb$-linearly.  This is the map that sends a point to the sum of lines through the given point. When  the points $p_1, \dots, p_n$ in $\PP$ and the lines $L_1, \dots, L_m$ in $\LL$ are enumerated, the map $A$ can be written as an $|\LL| \times |\PP|$ \emph{incidence matrix} whose $i,j$-entry is $1$ if $p_j \in L_i$ and is zero otherwise.  
Figure~\ref{fig:incmatrix} is the incidence matrix from the hypergraph in Figure \ref{fig:hypergraph}. 
We've chosen $L_1$, $L_2$, and $L_3$ to be the left, bottom, and right sides of the triangle, while taking $L_4$, $L_5$, and $L_6$, to be the blue, green, and red lines respectively. 

\begin{figure}[ht!]
\captionsetup[subfigure]{}
\begin{subfigure}[b]{.55\linewidth}
\begin{tikzpicture}[scale=0.7, baseline={([yshift=-.8ex]current bounding box.center)}]]
    \draw[ultra thick] (1,0) -- (4,5.196);
    \draw[ultra thick] (4,5.196) -- (7,0);
    \draw[ultra thick] (7,0) -- (1,0);
    \draw[ultra thick,bend left,blue,dashed,rounded corners] (1,0) to (2,1.732) to (6,1.732) to (5,0);
    \draw[ultra thick,bend left,green,dotted,rounded corners] (4,5.196) to (5,3.464) to (3,0) to (2,1.732);
    \draw[ultra thick,bend left,red,dashdotted,rounded corners] (7,0) to (5,0) to (3,3.464) to (5,3.464);
    \draw[thick,fill=blue] (1,0) circle (0.1);
    \draw[thick,fill=blue] (2,1.732) circle (0.1);
    \draw[thick,fill=blue] (3,3.464) circle (0.1);
    \draw[thick,fill=blue] (4,5.196) circle (0.1);
    \draw[thick,fill=blue] (5,3.464) circle (0.1);
    \draw[thick,fill=blue] (6,1.732) circle (0.1);
    \draw[thick,fill=blue] (7,0) circle (0.1);
    \draw[thick,fill=blue] (3,0) circle (0.1);
    \draw[thick,fill=blue] (5,0) circle (0.1);
    \draw(0.1,0.2) node {\small $p_1$};
    \draw(1.1,1.932) node {\small $p_2$};
    \draw(2.1,3.664) node {\small $p_3$};
    \draw(4,5.496) node {\small $p_4$};
    \draw(5.9,3.664) node {\small $p_9$};
    \draw(6.9,1.932) node {\small $p_8$};
    \draw(7.9,0.2) node {\small $p_7$};
    \draw(5,-0.6) node {\small $p_6$};
    \draw(3,-0.6) node {\small $p_5$};
\end{tikzpicture}
\caption{A $4$-uniform hypergraph.}
\label{fig:hypergraph}
\end{subfigure}
\begin{subfigure}[b]{.4\linewidth}
$\kbordermatrix{
&\bm{p_1}&\bm{p_2}&\bm{p_3}&\bm{p_4}&\bm{p_5}&\bm{p_6}&\bm{p_7}&\bm{p_8}&\bm{p_9}\\
\bm{L_1}&1&1&1&1&0&0&0&0&0\\
\bm{L_2}&1&0&0&0&1&1&1&0&0\\
\bm{L_3}&0&0&0&1&0&0&1&1&1\\
{\color{blue}\bm{L_4}}&1&1&0&0&0&1&0&1&0\\
{\color{green}\bm{L_5}}&0&1&0&1&1&0&0&0&1\\
{\color{red}\bm{L_6}}&0&0&1&0&0&1&1&0&1
}$
\subcaption{An incidence matrix.}
\label{fig:incmatrix}
\end{subfigure}
\caption{}
\label{fig:hypergraphmatrix}
\end{figure}
Since every line $L\in \LL$ is a collection of points from $\PP$ we can view $L$ as being in $\Zb\PP$ by summing the points that make up $L$.  This is just the transpose of the incidence map, $A^\T:\Zb\LL\rightarrow \Zb\PP$ which is $A^\T L=\sum_{p\in L} p$.  It turns out, determining whether or not $\Gamma$ can be made magic over some group is determined by  the image of $A^\T$ or, equivalently, the $\Zb$-row space of $A$.

Magic labelings of $\Gamma$ are functions on $\PP$, so it will be useful to consider $\Zb^\PP$ (\textit{resp.} $\Zb^\LL$), the space of $\Zb$-valued functions on $\PP$ (\textit{resp.} $\LL$).  It will be useful to extend $\Zb$-valued functions $f$ on $\PP$ to functions on $\Zb\PP$. We do this by
\[f\left(\sum_{p\in P}a_p p\right)=\sum_{p\in P}a_p f(p).\]
These functions can naturally be identified with $\Zb\PP$ (\textit{resp.} $\Zb\LL$) by the map
 \begin{eqnarray*}
 [\cdot]:\Zb^\PP&\rightarrow& \Zb\PP\\
 f&\mapsto & \sum_{p\in \PP} f(p) p
 \end{eqnarray*}
 and similarly for $[\cdot]:\Zb^\LL\rightarrow \Zb\LL$.  Depending on the context, it may be more useful to consider functions and at other times consider combinations of points and lines.

We usually will not be working over $\Zb$, but rather over abelian groups.  Since every abelian group $G$ is a $\Zb$-module, the notions of $\Zb\PP$ and $\Zb\LL$ extend naturally to $G\PP$ and $G\LL$, and we can view the incidence map as a $G$-map.  When we do this, we will denote the incidence map/matrix as $A_G$.  In the special case where $G=\Zb/r\Zb$ for some $r\in \Nb$ we will just write $A_r$.

A $G$-labeling of $\Gamma$ can be represented as a function $f:\PP \to G$.    Define $h\in \Zb^\LL$ by  \[[h]=A_G[f] = \sum_{{1\leq i \leq m}\atop {p \in L_i}} f(p)L_i\] as a function on the lines in $\LL$.  Then $h(L_i)=\sum_{p\in L_i} f(p)$.   From this perspective, we consider $f$ as a function on the lines by $f(A^\T L)=\sum_{p \in L} f(p)$ instead. 

We now introduce some terms  in order to get at the idea of ``magicness''.  A function $f:\PP \to G$ is called \emph{line-invariant} if there exists a constant $c \in G$ such that $f(A^\T L) = c$ for all $L \in \LL$.  Such a constant is called the \emph{magic constant} for $f$.  If, furthermore, we have that $f(p) \neq f(q)$ for all $p\neq q \in \PP$, then we say that $f$ is a \emph{magic} function.  A $k$-uniform hypergraph $\Gamma=(\PP, \LL)$ is said to be \emph{magic over} an abelian group $G$ 
if there exists a magic function $f:\PP \to G$.  Observe that since $k$-uniform hypergraphs have the same number of points on each line, every constant function $f(p)=c$, for any abelian group $G$ and any $c \in G$, is line-invariant.  A non-constant example is the 4-uniform hypergraph from Figure~\ref{fig:hypergraph}, which is magic over $(\Zb/2\Zb)^4$ with magic labeling given in Figure~\ref{fig:magicHG}.

\begin{figure}[ht!]
\captionsetup[subfigure]{labelformat=simple}
\centering
\begin{tikzpicture}[scale=0.7]
    \draw[ultra thick] (1,0) -- (4,5.196);
    \draw[ultra thick] (4,5.196) -- (7,0);
    \draw[ultra thick] (7,0) -- (1,0);
    \draw[ultra thick,bend left,blue,dashed,rounded corners] (1,0) to (2,1.732) to (6,1.732) to (5,0);
    \draw[ultra thick,bend left,green,dotted,rounded corners] (4,5.196) to (5,3.464) to (3,0) to (2,1.732);
    \draw[ultra thick,bend left,red,dashdotted,rounded corners] (7,0) to (5,0) to (3,3.464) to (5,3.464);
    \draw[thick,fill=blue] (1,0) circle (0.1);
    \draw[thick,fill=blue] (2,1.732) circle (0.1);
    \draw[thick,fill=blue] (3,3.464) circle (0.1);
    \draw[thick,fill=blue] (4,5.196) circle (0.1);
    \draw[thick,fill=blue] (5,3.464) circle (0.1);
    \draw[thick,fill=blue] (6,1.732) circle (0.1);
    \draw[thick,fill=blue] (7,0) circle (0.1);
    \draw[thick,fill=blue] (3,0) circle (0.1);
    \draw[thick,fill=blue] (5,0) circle (0.1);
    \draw(-0.1,0.2) node {\small $(1,1,1,1)$};
    \draw(0.9,1.932) node {\small $(1,1,0,1)$};
    \draw(1.9,3.664) node {\small $(1,1,1,0)$};
    \draw(4,5.596) node {\small $(1,1,0,0)$};
    \draw(6.1,3.664) node {\small $(0,1,1,0)$};
    \draw(7.1,1.932) node {\small $(1,0,0,1)$};
    \draw(8.1,0.2) node {\small $(0,0,1,1)$};
    \draw(5,-0.6) node {\small $(1,0,1,1)$};
    \draw(3,-0.6) node {\small $(0,1,1,1)$};
\end{tikzpicture}
\caption{A magic labeling over $(\Zb/2\Zb)^4$.}
\label{fig:magicHG}
\end{figure}
Line-invariant functions are closely tied to $\text{Ker} A_G$.  If $[h]=A_G[f]=0$ for some $f:\PP\rightarrow G$ then $f(A^\T L)=h(L)=0$ for all $L\in\LL$.  Therefore, $f$ is line-invariant with magic constant $0$.  The converse also holds, and so $\text{Ker} A_G$ is the space of all line-invariant functions with magic constant $0$.  In the case of $k$-uniform hypergraphs where $gcd(k, |G|)=1$ one can obtain all line invariant functions by adding the constant functions to $\text{Ker} A_G$.

\subsection*{Substructures}
The existence of certain substructures or subconfigurations within a hypergraph places restrictions on the potential line-invariant labelings.  When considering the unique $7_3$-configuration -- the Fano configuration --  removing any point and the three lines incident to that point results in what is called a complete quadrilateral (see Figure~\ref{fig:fano}).  A \emph{complete quadrilateral}, $\QQ=(\PP_\QQ, \LL_\QQ)$ is an arrangement of points and lines consisting of four lines, no three of which pass through the same point, and the six points of intersection of these lines. So $\QQ$ is a $(6_2,4_3)$-configuration in its own right.  Containing this substructure rules out many possible groups over which a magic labeling could exist due to the following  observation.

\begin{figure}[ht!]
\captionsetup[subfigure]{}
\centering
\begin{subfigure}[b]{.25\linewidth}
\begin{tikzpicture}[scale=0.8]
    \draw[ultra thick] (1,0) -- (5,0);
    \draw[ultra thick] (5,0) -- (3,3.464);
    \draw[ultra thick] (1,0) -- (3,3.464);
    \draw[ultra thick] (3,0) -- (3,3.464);
    \draw[ultra thick] (2,1.732) -- (5,0);
    \draw[ultra thick] (1,0) -- (4,1.732);
    \draw[ultra thick] (3,1.155) circle(1.155);
    \draw[thick,fill=blue] (1,0) circle (0.1);
    \draw[thick,fill=blue] (2,1.732) circle (0.1);
    \draw[thick,fill=blue] (3,0) circle (0.1);
    \draw[thick,fill=blue] (3,1.155) circle (0.1);
    \draw[thick,fill=blue] (3,3.464) circle (0.1);
    \draw[thick,fill=blue] (4,1.732) circle (0.1);
    \draw[thick,fill=blue] (5,0) circle (0.1);
\end{tikzpicture}
\subcaption{The Fano $7_3$-configuration.\\}
\end{subfigure}
\hskip 1.5cm
\begin{subfigure}[b]{.4\linewidth}
\begin{tikzpicture}[scale=0.8]
    \draw[ultra thick] (5,0) -- (3,3.464);
    \draw[ultra thick] (1,0) -- (3,3.464);
    \draw[ultra thick] (2,1.732) -- (5,0);
    \draw[ultra thick] (1,0) -- (4,1.732);
    \draw[thick,fill=blue] (1,0) circle (0.1);
    \draw[thick,fill=blue] (2,1.732) circle (0.1);
    \draw[thick,fill=blue] (3,1.155) circle (0.1);
    \draw[thick,fill=blue] (3,3.464) circle (0.1);
    \draw[thick,fill=blue] (4,1.732) circle (0.1);
    \draw[thick,fill=blue] (5,0) circle (0.1);
    \draw(0.5,0) node {$p$};
    \draw(5.5,0) node {$q$};
\end{tikzpicture}
\subcaption{The complete quadrilateral obtained by removing a point and its three incident lines.}
\end{subfigure}
\caption{}
\label{fig:fano}
\end{figure}

\begin{prop}\label{prop:quad}
Let $f:\PP \to G$ be any line-invariant function and let $\QQ$ be a complete quadrilateral in $\Gamma$.  If $p$ and $q$ are the two points in $\QQ$ that are not incident to a common line, then $2f(p)=2f(q)$.
\end{prop}
\begin{proof}
Let $L_1$, $L_2$ be the lines incident to $p$ and let $L_3$, $L_4$ be the lines incident to $q$.  Applying $f$ to the linear combination $L_1 + L_2 - L_3 - L_4$ we have
$$0 = f(L_1+L_2-L_3-L_4) = f(2p-2q) = 2f(p)-2f(q),$$
hence the result follows.
\end{proof}

For the Fano $7_3$-configuration specifically, we see as a corollary that, $2f(p)=2f(q)$ for all pairs of points $p,q \in \PP$ and all line-invariant functions $f$.  This is because given any pair of points $p$ and $q$ there exists a line connecting them and thus, by removing the third point on that line, we are left with a complete quadrilateral with $p$ and $q$ in the appropriate positions.  More generally though, any time such a structure exists within our hypergraph $\Gamma$, it follows that $\Gamma$ is not magic over any group which does not have 2 as a zero divisor
.  For the Fano $7_3$-configuration, it is known \cite{NN2} that the smallest group over which it admits a magic labeling is $(\Zb/2\Zb)^3$.

There are also substructures which forbid all line-invariant functions $f$ from being injective which we call \emph{forbidden substructures}.  Containing a forbidden substructure makes it impossible for a hypergraph to be magic over any abelian group.  More specifically, forbidden substructures create restrictions by admitting integer linear combinations of their lines which result in a difference of two points $p-q$, i.e.\ $p-q$ is in the image of $[\cdot]$.  For example, the smallest $n_3$-configuration which is not magic over any abelian group $G$ is the $9_3$-configuration which is denoted as configuration $(9_3)_3$ in \cite{G1} and is given in Figure~\ref{fig:(9_3)_3}.

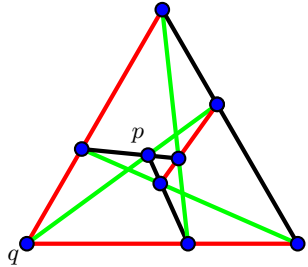
\begin{figure}[ht!]
\captionsetup[subfigure]{labelformat=simple}
\centering
\begin{tikzpicture}[scale=0.9]
\draw[ultra thick, red] (0,0) -- (2,3.46);
\draw[ultra thick, red] (0,0) -- (4, 0);
\draw[ultra thick, green] (0,0) -- (2.81, 2.06);
\draw[ultra thick, green] (0.81, 1.4) -- (4,0);
\draw[ultra thick] (0.81, 1.4) -- (2.24, 1.26);
\draw[ultra thick] (1.79, 1.31) -- (2.38, 0);
\draw[ultra thick, green] (2,3.46) -- (2.38, 0);
\draw[ultra thick, red] (1.97, 0.89) -- (2.81, 2.06);
\draw[ultra thick] (2, 3.46) -- (4, 0);
\draw[thick,fill=blue] (0,0) circle (0.1);
\draw[thick,fill=blue] (4,0) circle (0.1);
\draw[thick,fill=blue] (2, 3.46) circle (0.1);
\draw[thick,fill=blue] (2.38, 0) circle (0.1);
\draw[thick,fill=blue] (2.81, 2.06) circle (0.1);
\draw[thick,fill=blue] (0.81, 1.4) circle (0.1);
\draw[thick,fill=blue] (2.24, 1.26) circle (0.1);
\draw[thick,fill=blue] (2.81, 2.06) circle (0.1);
\draw[thick,fill=blue] (1.79, 1.31) circle (0.1);
\draw[thick,fill=blue] (1.97, 0.89) circle (0.1);
\draw(-0.2,-0.2) node {\small $q$};
\draw(1.64,1.61) node {\small $p$};
\end{tikzpicture}
\caption{The smallest configuration which is not magic over any abelian group.}
\label{fig:(9_3)_3}
\end{figure}

Observe, that if we take the sum of the green lines and subtract the three red lines then we are left with exactly $p-q$.  Hence, if $f \colon \mathcal{P} \to G$ is any line-invariant function, then applying $f$ to this linear combination shows us that $f(p)=f(q)$.  We reproduce this subconfiguration more clearly (replacing green lines with black ones) in Figure~\ref{fig:subconfig1} together with another in Figure~\ref{fig:subconfig2}.  We have found that the majority of $n_3$-configurations for $n=11$, 12, 13, and 14 contain at least one of these two small examples of forbidden substructures.  
In each case, we get the difference $p-q$ by adding the black lines and subtracting the red ones.  

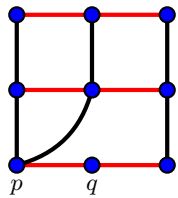
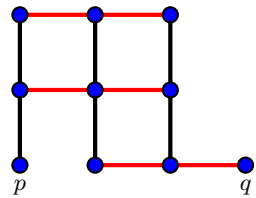
\begin{figure}[ht!]
\captionsetup[subfigure]{}
\centering
\begin{subfigure}[b]{.28\linewidth}
\begin{tikzpicture}
    \draw[ultra thick,red] (1,0) -- (3,0);
    \draw[ultra thick,red] (1,1) -- (3,1);
    \draw[ultra thick,red] (1,2) -- (3,2);
    \draw[ultra thick] (1,0) -- (1,2);
    \draw[ultra thick,bend right] (1,0) to (2,1);
    \draw[ultra thick] (2,1) -- (2,2);
    \draw[ultra thick] (3,0) -- (3,2);
    \draw[thick,fill=blue] (1,0) circle (0.1);
    \draw[thick,fill=blue] (2,0) circle (0.1);
    \draw[thick,fill=blue] (3,0) circle (0.1);
    \draw[thick,fill=blue] (1,1) circle (0.1);
    \draw[thick,fill=blue] (2,1) circle (0.1);
    \draw[thick,fill=blue] (3,1) circle (0.1);
    \draw[thick,fill=blue] (1,2) circle (0.1);
    \draw[thick,fill=blue] (2,2) circle (0.1);
    \draw[thick,fill=blue] (3,2) circle (0.1);
    \draw(1,-0.3) node {\small $p$};
    \draw(2,-0.3) node {\small $q$};
    \draw[white] (0.5,0) node{a}; 
\end{tikzpicture}
\subcaption{A forbidden substructure derived from a $9_3$-configuration.}
\label{fig:subconfig1}
\end{subfigure}
\hskip 1.5cm
\begin{subfigure}[b]{.28\linewidth}
\begin{tikzpicture}
    \draw[ultra thick,red] (2,0) -- (4,0);
    \draw[ultra thick,red] (1,1) -- (3,1);
    \draw[ultra thick,red] (1,2) -- (3,2);
    \draw[ultra thick] (1,0) -- (1,2);
    \draw[ultra thick] (2,0) -- (2,2);
    \draw[ultra thick] (3,0) -- (3,2);
    \draw[thick,fill=blue] (1,0) circle (0.1);
    \draw[thick,fill=blue] (2,0) circle (0.1);
    \draw[thick,fill=blue] (3,0) circle (0.1);
    \draw[thick,fill=blue] (1,1) circle (0.1);
    \draw[thick,fill=blue] (2,1) circle (0.1);
    \draw[thick,fill=blue] (3,1) circle (0.1);
    \draw[thick,fill=blue] (1,2) circle (0.1);
    \draw[thick,fill=blue] (2,2) circle (0.1);
    \draw[thick,fill=blue] (3,2) circle (0.1);
    \draw[thick,fill=blue] (4,0) circle (0.1);
    \draw(1,-0.3) node {\small $p$};
    \draw(4,-0.3) node {\small $q$};
\end{tikzpicture}
\subcaption{Another forbidden substructure forcing $f(p)=f(q)$.}
\label{fig:subconfig2}
\end{subfigure}
\caption{Forbidden substructures appearing commonly in $n_3$-configurations.}
\label{fig:subconfigs}
\end{figure}

In the setting of $k$-uniform hypergraphs, these types of forbidden substructures can be generalized and there are many other types of forbidden substructures as well.  For example, any pair of lines sharing $k-1$ common points forces the two additional points to be equal.  See Figure~\ref{fig:substructure1} for this example and Figures~\ref{fig:substructure2} and \ref{fig:substructure3} generalizations of the previous two.

\begin{figure}[ht!]
\captionsetup[subfigure]{}
\centering
\begin{subfigure}[b]{.28\linewidth}
\begin{tikzpicture}[scale=0.75,baseline={([yshift=-2ex]current bounding box.center)}]
    \draw[ultra thick,red,bend left] (0,0) to (1,1);
    \draw[ultra thick,red,bend right] (1,1) to (2,2);
    \draw[ultra thick,bend right] (0,0) to (1,1);
    \draw[ultra thick,bend left] (1,1) to (2,2);
    \draw[dotted,thick] (2,2) to (3,3);
    \draw[ultra thick,red] (3,3) -- (4,3);
    \draw[ultra thick] (3,3) -- (3,4);
    \draw[thick,fill=blue] (0,0) circle (0.133);
    \draw[thick,fill=blue] (1,1) circle (0.133);
    \draw[thick,fill=blue] (2,2) circle (0.133);
    \draw[thick,fill=blue] (3,3) circle (0.133);
    \draw[thick,fill=blue] (3,4) circle (0.133);
    \draw[thick,fill=blue] (4,3) circle (0.133);
    \draw(2.6,4) node {\small $p$};
    \draw(4,2.6) node {\small $q$};
\end{tikzpicture}
\subcaption{A forbidden substructure.}
\label{fig:substructure1}
\end{subfigure}
\qquad
\begin{subfigure}[b]{.28\linewidth}
\begin{tikzpicture}[scale=1.5,baseline={([yshift=-.8ex]current bounding box.center)}]
    \draw[ultra thick,red] (0,2) -- (0.67,2);
    \draw[thick,red,dotted] (0.67,2) -- (1.5,2);
    \draw[ultra thick,red] (1.33,2) -- (1.67,2);
    \draw[ultra thick,red] (0,1.67) -- (0.67,1.67);
    \draw[thick,red,dotted] (0.67,1.67) -- (1.5,1.67);
    \draw[ultra thick,red] (1.33,1.67) -- (1.67,1.67);
    \draw[ultra thick,red] (0,1.33) -- (0.67,1.33);
    \draw[thick,red,dotted] (0.67,1.33) -- (1.5,1.33);
    \draw[ultra thick,red] (1.33,1.33) -- (1.67,1.33);
    \draw[ultra thick,red] (0,0.5) -- (0.67,0.5);
    \draw[thick,red,dotted] (0.67,0.5) -- (1.5,0.5);
    \draw[ultra thick,red] (1.33,0.5) -- (1.67,0.5);
    \draw[ultra thick,red] (0,0) -- (0.67,0);
    \draw[thick,red,dotted] (0.67,0) -- (1.5,0);
    \draw[ultra thick,red] (1.33,0) -- (1.67,0);
    \draw[ultra thick] (0,1.33) -- (0,2);
    \draw[dotted,thick] (0,0.5) -- (0,1.3);
    \draw[ultra thick] (0,0) -- (0,0.5);
    \draw[thick,fill=blue] (0,0) circle (0.0667);
    \draw[thick,fill=blue] (0,0.5) circle (0.0667);
    \draw[thick,fill=blue] (0,1.33) circle (0.0667);
    \draw[thick,fill=blue] (0,1.67) circle (0.0667);
    \draw[thick,fill=blue] (0,2) circle (0.0667);
    \draw[ultra thick] (0.33,1.33) -- (0.33,2);
    \draw[dotted,thick] (0.33,0.5) -- (0.33,1.33);
    \draw[ultra thick,bend right] (0,0) to (0.33,0.5);
    \draw[thick,fill=blue] (0.33,0) circle (0.0667);
    \draw[thick,fill=blue] (0.33,0.5) circle (0.0667);
    \draw[thick,fill=blue] (0.33,1.33) circle (0.0667);
    \draw[thick,fill=blue] (0.33,1.67) circle (0.0667);
    \draw[thick,fill=blue] (0.33,2) circle (0.0667);
    \draw[ultra thick] (0.67,1.3) -- (0.67,2);
    \draw[dotted,thick] (0.67,0.5) -- (0.67,1.33);
    \draw[ultra thick] (0.67,0) -- (0.67,0.5);
    \draw[thick,fill=blue] (0.67,0) circle (0.0667);
    \draw[thick,fill=blue] (0.67,0.5) circle (0.0667);
    \draw[thick,fill=blue] (0.67,1.33) circle (0.0667);
    \draw[thick,fill=blue] (0.67,1.67) circle (0.0667);
    \draw[thick,fill=blue] (0.67,2) circle (0.0667);
    \draw[ultra thick] (1.33,1.33) -- (1.33,2);
    \draw[dotted,thick] (1.33,0.5) -- (1.33,1.3);
    \draw[ultra thick] (1.33,0) -- (1.33,0.5);
    \draw[thick,fill=blue] (1.33,0) circle (0.0667);
    \draw[thick,fill=blue] (1.33,0.5) circle (0.0667);
    \draw[thick,fill=blue] (1.33,1.33) circle (0.0667);
    \draw[thick,fill=blue] (1.33,1.67) circle (0.0667);
    \draw[thick,fill=blue] (1.33,2) circle (0.0667);
    \draw[ultra thick] (1.67,1.33) -- (1.67,2);
    \draw[dotted,thick] (1.67,0.5) -- (1.67,1.3);
    \draw[ultra thick] (1.67,0) -- (1.67,0.5);
    \draw[thick,fill=blue] (1.67,0) circle (0.0667);
    \draw[thick,fill=blue] (1.67,0.5) circle (0.0667);
    \draw[thick,fill=blue] (1.67,1.33) circle (0.0667);
    \draw[thick,fill=blue] (1.67,1.67) circle (0.0667);
    \draw[thick,fill=blue] (1.67,2) circle (0.0667);
    \draw(0,-0.3) node {\small $p$};
    \draw(0.33,-0.3) node {\small $q$};
    \draw[white] (-0.2,0) node {a}; 
\end{tikzpicture}
\subcaption{A generalization of Figure~\ref{fig:subconfig1}.}
\label{fig:substructure2}
\end{subfigure}
\qquad
\begin{subfigure}[b]{.28\linewidth}
\begin{tikzpicture}[scale=1.5,baseline={([yshift=-.8ex]current bounding box.center)}]
    \draw[ultra thick,red] (0,2) -- (0.67,2);
    \draw[thick,red,dotted] (0.67,2) -- (1.5,2);
    \draw[ultra thick,red] (1.33,2) -- (1.67,2);
    \draw[ultra thick,red] (0,1.67) -- (0.67,1.67);
    \draw[thick,red,dotted] (0.67,1.67) -- (1.5,1.67);
    \draw[ultra thick,red] (1.33,1.67) -- (1.67,1.67);
    \draw[ultra thick,red] (0,1.33) -- (0.67,1.33);
    \draw[thick,red,dotted] (0.67,1.33) -- (1.5,1.33);
    \draw[ultra thick,red] (1.33,1.33) -- (1.67,1.33);
    \draw[ultra thick,red] (0,0.5) -- (0.67,0.5);
    \draw[thick,red,dotted] (0.67,0.5) -- (1.5,0.5);
    \draw[ultra thick,red] (1.33,0.5) -- (1.67,0.5);
    \draw[ultra thick,red] (0.33,0) -- (0.67,0);
    \draw[thick,red,dotted] (0.67,0) -- (1.5,0);
    \draw[ultra thick,red] (1.33,0) -- (2,0);
    \draw[ultra thick] (0,1.33) -- (0,2);
    \draw[dotted,thick] (0,0.5) -- (0,1.3);
    \draw[ultra thick] (0,0) -- (0,0.5);
    \draw[thick,fill=blue] (0,0) circle (0.0667);
    \draw[thick,fill=blue] (0,0.5) circle (0.0667);
    \draw[thick,fill=blue] (0,1.33) circle (0.0667);
    \draw[thick,fill=blue] (0,1.67) circle (0.0667);
    \draw[thick,fill=blue] (0,2) circle (0.0667);
    \draw[ultra thick] (0.33,1.33) -- (0.33,2);
    \draw[dotted,thick] (0.33,0.5) -- (0.33,1.33);
    \draw[ultra thick] (0.33,0) to (0.33,0.5);
    \draw[thick,fill=blue] (0.33,0) circle (0.0667);
    \draw[thick,fill=blue] (0.33,0.5) circle (0.0667);
    \draw[thick,fill=blue] (0.33,1.33) circle (0.0667);
    \draw[thick,fill=blue] (0.33,1.67) circle (0.0667);
    \draw[thick,fill=blue] (0.33,2) circle (0.0667);
    \draw[ultra thick] (0.67,1.3) -- (0.67,2);
    \draw[dotted,thick] (0.67,0.5) -- (0.67,1.33);
    \draw[ultra thick] (0.67,0) -- (0.67,0.5);
    \draw[thick,fill=blue] (0.67,0) circle (0.0667);
    \draw[thick,fill=blue] (0.67,0.5) circle (0.0667);
    \draw[thick,fill=blue] (0.67,1.33) circle (0.0667);
    \draw[thick,fill=blue] (0.67,1.67) circle (0.0667);
    \draw[thick,fill=blue] (0.67,2) circle (0.0667);
    \draw[ultra thick] (1.33,1.33) -- (1.33,2);
    \draw[dotted,thick] (1.33,0.5) -- (1.33,1.3);
    \draw[ultra thick] (1.33,0) -- (1.33,0.5);
    \draw[thick,fill=blue] (1.33,0) circle (0.0667);
    \draw[thick,fill=blue] (1.33,0.5) circle (0.0667);
    \draw[thick,fill=blue] (1.33,1.33) circle (0.0667);
    \draw[thick,fill=blue] (1.33,1.67) circle (0.0667);
    \draw[thick,fill=blue] (1.33,2) circle (0.0667);
    \draw[ultra thick] (1.67,1.33) -- (1.67,2);
    \draw[dotted,thick] (1.67,0.5) -- (1.67,1.3);
    \draw[ultra thick] (1.67,0) -- (1.67,0.5);
    \draw[thick,fill=blue] (1.67,0) circle (0.0667);
    \draw[thick,fill=blue] (1.67,0.5) circle (0.0667);
    \draw[thick,fill=blue] (1.67,1.33) circle (0.0667);
    \draw[thick,fill=blue] (1.67,1.67) circle (0.0667);
    \draw[thick,fill=blue] (1.67,2) circle (0.0667);
    \draw[thick,fill=blue] (2,0) circle (0.0667);
    \draw(0,-0.3) node {\small $p$};
    \draw(2,-0.3) node {\small $q$};
\end{tikzpicture}
\subcaption{A generalization of Figure~\ref{fig:subconfig2}.}
\label{fig:substructure3}
\end{subfigure}
\caption{}
\label{fig:substructures}
\end{figure}

Searching for particular substructures within complicated hypergraphs is computationally time consuming.  Furthermore, it may not provide definitive information as 
it is unclear whether there is a finite list of forbidden substructures which accounts for all $k$-uniform hypergraphs which are never magic.  In fact, we do not even know if there is such a finite list in the special case of $n_3$-configurations. 
With that in mind, it would be helpful to have a definitive, computational way to determine whether a given $k$-uniform hypergraph can be magically labeled over some abelian group.

\subsection*{When $\bm{\Gamma}$ is Not Magic}

If we can find a pair of points $p$ and $q$ that have equal values for all line-invariant functions, then our $k$-uniform hypergraph cannot be magic over any group.  The somewhat surprising result is that the converse is also true, i.e.\ if a $k$-uniform hypergraph cannot be magic over any group, then there must exist a fixed pair of points $p$ and $q$ that have equal values under \emph{every} line-invariant function.  To see this, suppose that no such pair of points exists.  Then for every pair of points $p,q \in \PP$ there exists some abelian group $G_{p,q}$ and some line-invariant function $f_{p,q}:\PP \to G_{p,q}$ such that $f_{p,q}(p)  \neq f_{p,q}(q)$.  Thus, the function
$$f=\bigoplus_{(p,q)\in\PP} f_{p,q}:\PP \to \bigoplus_{(p,q) \in \PP} G_{p,q}$$
is both line-invariant and distinguishes all points.  Hence, we have found a group over which the configuration $\Gamma$ is magic.

\begin{theorem}\label{thm:equalpair} A $k$-uniform hypergraph $\Gamma=(\PP,\LL)$ is not magic over any abelian group if and only if there exists a pair of points $p,q \in \PP$ such that, for all abelian groups $G$ and all line-invariant functions $f:\PP \to G$, we have $f(p)=f(q)$.
\end{theorem}
This result is theoretical in nature, but it is the basis for a result  which allows one to computationally determine whether or not a $k$-uniform hypergraph can be made magic.  This is done by showing  $f(p)=f(q)$ for every line invariant function if and only if  $p-q\in \im A^\T$.

Showing that a hypergraph cannot be made magic is relatively straight forward. Assume $p, q\in \PP$ with $p-q\in \im A^\T$.  Then there exists a $\Zb$-linear combination of lines, $S$, such that $A^\T S=A^\T\sum_{i=1}^m b_iL_i=p-q$.  Next, let $f$ be a line invariant function to some group $G$ with magic constant $c\in G$.  It follows that $f(A^\T S)=f(p-q)=f(p)-f(q)$.  However, \[f(A^\T S)=\sum_{i=1}^m b_if(A^\T L_i)=c\sum_{i=1}^m b_i.\]
So, $c\sum_{i=1}^m b_i=f(p)-f(q)$.  If we can show $\sum_{i=1}^m b_i=0$ then we will have $f(p)=f(q)$ and the function cannot be magic.

Consider the $\Zb$-linear map $\ast:\Zb\PP\rightarrow \Zb$, where \[\left(\sum_{i=1}^n a_i p_i\right)^\ast=\sum_{i=1}^n a_i.\]
This function is counting, with multiplicity, how many points are in a given linear combination.  For a $k$-uniform hypergraph and any line $L$ we have $[A^\T L]^\ast=k$  since it is just counting how many points are on the line.  Therefore,
\[0=(p-q)^*=(A^\T S)^*=\sum_{i=1}^m b_i(A^\T L_i)^*=k \sum_{i=1}^m b_i,\]
and so $\sum_{i=1}^m b_i=0$ as desired.  This is summarized in the following:

\begin{prop}\label{prop:notmagic}
 A $k$-uniform hypergraph $\Gamma=(\PP, \LL)$ is not magic over any Ableian group $G$ if there are $p, q\in \PP$ with $p-q \in \im A^\T$.   	
\end{prop}

Since $\im A^\T$ is the $\Zb$-row space of $A$, this is equivalent to having $e_i-e_j$ in the row space of the incidence matrix for some $i \neq j$, where $e_i$ and $e_j$ are standard basis vectors.  It is more difficult to show that if $\Gamma$ is not magic over any abelian group then there must exist a $p, q\in \PP$ with $p-q\in \im A^\T$.  To get at this result we solve the problem modulo $m$ for each $m\in \Nb$ ($m>1$) and then show a solution exists over $\Zb$.  This technique is known as a local-global principle.

\subsection*{Local-Global Principles}
Local-global principles play an important role in number theory.  The basic idea is that if one wants to guarantee a solution to an equation defined over $\Zb$, one only needs to show a solution modulo $r$ for all $r\in \Nb $ with $r>1$.  It should be noted that this concept often extends to other rings, but we focus here on the results over $\Zb$.

For a linear equation in one variable this result follows from the Chinese Remainder Theorem, but the result is less obvious in other situations.  The local-global principle holds for any single variable polynomial equation over $\Zb$.  However, it does not always hold for multi-variable equations.  For quadratic equations in multiple variables, this result is known as the Hasse-Minkowski theorem, and is the first major result of this kind. For a treatment of this result in English see \cite{Cohen}.  Unfortunately, a local-global principle is often too much to ask for in many settings and breaks down even for cubic equations in multiple variables.  Selmer \cite{Selmer} showed that the equation $3x^3+4y^3+5z^3=0$ has non-zero solutions $\text{mod } r$ and over $\Rb$, but no non-zero solutions over $\Zb$.

In this paper we will be applying a local-global principle for systems of linear equations.  We provide a simplified version of the result below.

\begin{theorem}[\cite{LQ}, Section 2.3]
	For a matrix $M$ and column vector $y$ defined over $\Zb$, if the system $Mx=y$ has a solution $\text{mod } r$ for all $r\in \Nb$, $r>1$, then $Mx=y$ has a solution over $\Zb$.	
\end{theorem}

\subsection*{Some Module Algebra}

The local-global principle states that if we want to know if $p-q\in\im(A^\T)$ we can study the same question for $\im(A_r^\T)$ for $r>1$ instead. A classic result from linear algebra (see e.g.\ \cite{Axler}, Proposition~6.46) states that $\im M^\T = (\ker M)^\perp$ for any matrix defined over a field. However, this result breaks down for matrices defined over a ring, even one as nice as $\Zb$.

For example, consider the matrix
\[M=\left(\begin{array}{cc}
2 & 0\\
0 & 0
\end{array}\right).\]
For this matrix, $\im M^\T$ is spanned by $[2,0]$, $\ker M$ is spanned by $[0,1]$, and $(\ker M)^\perp$ is spanned by $[1,0]$ and hence  $\im M^\T\neq  (\ker M)^\perp$ with respect to the standard inner product.

The classic result that $\im M^\T=  (\ker M)^\perp$ over a field relies on the fact that for a subspace $W$ of a finite dimensional vector space $V$ we have $(W^\perp)^\perp = W$ with respect to an inner product.  This does not hold over most rings.  However, it does hold for $\Zb/r\Zb$.  We have not been able to find a proof in the literature, so we provide one here using the characters of the group $\Zb/r\Zb$.\footnote{The idea for this argument was suggested in response to a question on MathOverflow, see \cite{MO}.}

Working over $\Zb/r\Zb$ for any integer $r>1$, observe that for each vector $\vec{v} = (v_1, v_2, \dots, v_s) \in (\Zb/r\Zb)^s$ we have a cyclotomic character $\chi_{\vec{v}}:(\Zb/r\Zb)^s \to \Cb^\times$ defined by setting \[\chi_{\vec{v}}(x_1, \dots, x_s) = e^{\frac{2\pi i}{n}(v_1x_1 + v_2x_2 + \dots + v_sx_s)}.\]  It follows that two vectors $\vec{v}$ and $\vec{x}$ are orthogonal in $(\Zb/r\Zb)^s$ if and only if $\chi_{\vec{v}}(\vec{x}) = 1$.  For any vector $\vec{x} \in (\Zb/r\Zb)^s$ and any character $\chi$ there is also a natural pairing $(\vec{x},\chi) = \chi(\vec{x})$.  Given any subspace $H$, we define $H^\downfree = \{\chi \mid \chi(\vec{x})=1 \text{ for all }\vec{x} \in H\}$.  Under this pairing, it has been shown (see \cite[Proposition 3.4]{Washington}) that $(H^\downfree)^\downfree = H$ for all subspaces $H \subseteq (\Zb/r\Zb)^s$.  With a one-to-one correspondence between cyclotomic characters and vectors and the fact that $\langle \vec{v},\vec{x} \rangle = 0$ if and only if $\langle \vec{x}, \chi_{\vec{v}} \rangle=1$, it follows that $(V^\perp)^\perp=V$ for all subspaces $V$ when using the standard inner product as well.  This proves the following proposition:

\begin{prop}\label{prop:mperp}
	For a $\Zb/r\Zb$-map $M$, $\im M^\T=  (\ker M)^\perp$.
\end{prop}


\subsection*{When $\bm{\Gamma}$ \emph{is} Magic}

We have already seen that a $k$-uniform hypergraph $\Gamma$ is not magic over any group if $p-q\in \im A^\T$.  We now consider the situation where we know $\Gamma$ is not magic over any group and determine what we can say about $\im A^\T$.

Theorem~\ref{thm:equalpair} implies that there exists some pair of points $p, q \in \PP$ such that $f(p)=f(q)$ for all line-invariant functions for all groups $G$.   Since $\ker A_G$ contains all line-invariant functions with magic constant $0$, it follows specifically that $f(p)=f(q)$ for every function $f \in \ker A_G$.

The standard inner product on $\Zb\PP^2$ (\textit{resp. $(\Zb/r\Zb)\PP^2$}) with respect to the basis of points, yields the pairing
\[\langle p-q, [f]\rangle = f(p-q)= f(p)-f(q)=0.\]
So $p-q\in (\ker A)^\perp$ (\textit{resp.} $(\ker A_r)^\perp$).  When working over $\Zb$, as we have noted, it is not necessarily true that $(\ker A)^\perp=\im A^\T$.  However, it is true that $(\ker A_r)^\perp=\im A_r^\T$, and $p-q\in \im A_r^\T$ for all $r\in \Nb, r>1$.  Therefore, by the local-global principle for systems of linear equations $p-q\in \im A^\T$.  This gives our main result.


\begin{main} A $k$-uniform hypergraph $\Gamma=(\PP, \LL)$ with incidence matrix $A$ is not magic over any group $G$ if and only if there are $p, q\in \PP$ with $p-q \in \im A^\T$.
\end{main}

This result has a geometric interpretation as well. Assume $p-q$ is in the $\Zb$-row space of $A$.  Then $p-q$ can be written as a linear combination of lines (rows), and the lines with non-zero coefficients correspond to a forbidden substructure.  
For example, in the case of the $14_3$-configuration with points labeled $1, \ldots, 14$ and lines described as in Figure \ref{fig:14_3lines},
\begin{figure}[h!]
\[\begin{array}{cccccccccccccc}
L_1&L_2&L_3&L_4&L_5&L_6&L_7&L_8&L_9&L_{10}&L_{11}&L_{12}&L_{13}&L_{14}\\
1~&1~&1~&2~&5~&2~&8~&5~&3~&7~&6~&3~&9~&11\\
2&4&6&4&6&7&9&10&11&8&12&4&12&13\\
3&5&7&8&9&10&10&11&12&13&13&14&14&14\end{array}\]
\caption{Lines of a $14_3$-configuration.}
\label{fig:14_3lines}
\end{figure}
there is only one $\Zb$-linear combination of the lines that results in $\vec{e}_i-\vec{e}_j$ for some $i<j$, namely the combination
$$-L_1 + L_2 + L_6 + L_7 - 2L_8 + L_9 - L_{10} - L_{13} + L_{14} = \vec{e}_4 - \vec{e}_5.$$
The lines that appear here form the subconfiguration given in Figure~\ref{fig:newsubconfig}.  By taking the black lines to have coefficient 1, the red lines to have coefficient -1, and the green line to have coefficient -2, this subconfiguration forces $f(p)$ to equal $f(q)$ for all line-invariant functions.

\begin{figure}[ht!]
\captionsetup[subfigure]{labelformat=simple}
\centering
\begin{tikzpicture}
    \draw[ultra thick] (1,0) -- (4,0);
    \draw[ultra thick] (1,1) -- (4,1);
    \draw[ultra thick] (0,2) -- (2,2);
    \draw[ultra thick] (1,1) -- (3,2);
    \draw[ultra thick] (1,0) -- (4,-1);
    \draw[ultra thick,green] (1,0) -- (1,2);
    \draw[ultra thick,red] (2,0) -- (2,2);
    \draw[ultra thick,red] (3,-0.667) -- (3,2);
    \draw[ultra thick,red] (4,-1) -- (4,1);
    \draw[thick,fill=blue] (1,0) circle (0.1);
    \draw[thick,fill=blue] (2,0) circle (0.1);
    \draw[thick,fill=blue] (3,-0.667) circle (0.1);
    \draw[thick,fill=blue] (1,1) circle (0.1);
    \draw[thick,fill=blue] (2,2) circle (0.1);
    \draw[thick,fill=blue] (3,1) circle (0.1);
    \draw[thick,fill=blue] (1,2) circle (0.1);
    \draw[thick,fill=blue] (2,1.5) circle (0.1);
    \draw[thick,fill=blue] (3,2) circle (0.1);
    \draw[thick,fill=blue] (0,2) circle (0.1);
    \draw[thick,fill=blue] (4,-1) circle (0.1);
    \draw[thick,fill=blue] (4,0) circle (0.1);
    \draw[thick,fill=blue] (4,1) circle (0.1);
    \draw(0,2.3) node {\small $p$};
    \draw(1,2.3) node {\small $q$};
\end{tikzpicture}
\caption{A forbidden substructure.}
\label{fig:newsubconfig}
\end{figure}

\subsection*{When $\bm{\Gamma}$ is Magic over $\bm{\Zb}$}

A slight adjustment to the ideas behind our main theorem allows us to move back towards the more traditional magic setting and determine whether a $k$-uniform hypergraph is magic over $\Zb$ as well.  The result we give in Theorem~\ref{thm:magicoverZ} might be known in the dual setting of edge labelings of graphs (see \cite{Doob2}), however it is stated in a much different way and the proof appears to be lacking crucial details.  Before we give this result, we first need an analogous version of Theorem~\ref{thm:equalpair} specifically for the group $\Zb$.

The idea is similar to our previous proof, but instead of taking a direct sum of labelings with values in a direct product of groups, we scale labelings in $\Zb$ so that we may take an internal sum of labelings that will end up being injective.

Consider a $k$-uniform hypergraph $\Gamma=(\PP,\LL)$ and suppose that for each $p, q\in \PP$ there is a line invariant function $f_{p,q}: \PP \to \Zb$ such that $f_{p,q}(p) \neq f_{p,q}(q)$. Given the points $p_1, \dots, p_n \in \PP$ we let $b=\max\{|f_{p,q}(p_i)| \mid p,q,p_i \in \PP\}$.  Then we define a new line-invariant function 
\[f=\sum_{1 \leq i < j \leq n} (2b+1)[n(i-1) + j]f_{p_i,p_j}\]
By construction $f$ is injective and $\Gamma$ is magic over $\Zb$. In summary

\begin{lemma}\label{lemma:Zequalpair}
A $k$-uniform hypergraph $\Gamma=(\PP,\LL)$ is not magic over $\Zb$ if and only if there exists a pair of points $p,q \in \PP$ such that for all line-invariant functions $f:\PP \to \Zb$ we have $f(p)=f(q)$.
\end{lemma}
We can use this result in determining whether or not a hypergraph is magic over $\Zb$.  The result is similar to our main theorem, except now we are interested in whether or not $p-q$ is in $\im A_\Qb^\T$.  If $p-q$ is in $\im A_\Qb^\T$ then we can write
\[p-q =\sum_{i=1}^m b_i A_\Qb^\T L_i\]
with $b_i\in \Qb$. Let $d$ be the common denominator of the $b_i$.  Then by multiplying through we see that $dp-dq$ is a $\Zb$-linear combination of the lines.  If $f$ is a $\Zb$-line-invariant function with magic constant $c$ then we have
\[ =f\left(\sum_{i=1}^m db_i A_\Qb^\T L_i\right)=\sum_{i=1}^m db_i f\left(A_\Qb^\T L_i\right)=\sum_{i=1}^m db_i c=c\sum_{i=1}^m db_i.\]
This sum is $0$ for the same reasons as in the main theorem.  Similarly $f(dp -dq)=df(p)-d(q)$, and combining these we obtain $df(p)=df(q)$.  Since $f$ takes values in $\Zb$ this implies $f(p)=f(q)$ and the function is not magic.

Conversely, if $\Gamma$ is not magic over $\Zb$, then by Lemma~\ref{lemma:Zequalpair} it follows that there exists some pair of points $p, q \in \PP$ such that $f(p)=f(q)$ for all line-invariant functions $f:\PP \to \Zb$.  Just as in the proof of the Main Theorem, this puts $p-q \in (\ker A_\Qb)^\perp$.  Note, however, that since we are now working over the field $\Qb$ we may immediately conclude that $(\ker A_\Qb)^\perp = \im A_\Qb^\T$.   We summarize the results below.

\begin{theorem}\label{thm:magicoverZ} A $k$-uniform hypergraph $\Gamma=(\PP,\LL)$ with incidence matrix $A$ is not magic over $\Zb$ if and only if there exists some $\Qb$-linear combination of the rows of $A$ that is equal to $\vec{e}_i - \vec{e}_j$ for some $i \neq j$.
\end{theorem}

\subsection*{Applications to $n_3$-configurations}
By applying our algorithms to the known complete lists of $n_3$-configurations for $n=7, 8, \dots, 14$, we are able to determine how many of these configurations are magic over some abelian group as well as how many are magic over $\Zb$ specifically.  In addition, it is important to point out that, for $n_3$-configurations, the two forbidden substructures in Figure~\ref{fig:subconfigs} are the only forbidden substructures using six or fewer lines.  As mentioned previously, these substructures account for the majority of $n_3$-configurations, for $n=7, \dots, 14$, which are not magic.  In the table below, we also count the number of configurations containing at least one of these small substructures.   

$$
\begin{tabular}{||c|c|c|c|c|c||}
\hline \hline
$n$ & \# Configs. & \# w/ Small & \# Magic & Magic/Total & \# Magic over $\Zb$\\
\hline
7 & 1 & 0 & 1 &1& 0\\
\hline
8 & 1 & 0 & 1 &1& 0\\
\hline
9 & 3 & 1 & 2 &2/3& 0\\
\hline
10 & 10 & 5 & 5 &1/2& 0\\
\hline
11 & 31 & 22 & 9 &$\approx$0.290& 1\\
\hline
12 & 229 & 187 & 34 &$\approx$0.148& 1\\
\hline
13 & 2036 & 1727 & 198 & $\approx$0.097 & 22\\
\hline
14 & 21399 & 17933 & 1467 & $\approx$0.069 &  125\\
\hline \hline
\end{tabular}
$$

Based on our findings in the table, it appears as though the fraction of magic $n_3$-configurations within the total number of configurations is approaching zero.  
This seems reasonable from the perspective of forbidden substructures.  Observe that the small forbidden substructures found commonly in $n_3$-configurations (see Figure~\ref{fig:subconfigs}) involve collections of parallel lines.  In addition, as $n$ grows larger it becomes easier to find parallel lines within $n_3$-configurations.  Furthermore, there is an increased chance that an $n_3$-configuration will contain larger forbidden substructures as well.

\begin{conj} $\displaystyle \lim_{n\to\infty} \frac{\text{\# of $n_3$-configs magic over some group}}{\text{\# of $n_3$-configs}} \to 0$.
\end{conj}

The smallest example of an $n_3$-configuration that is magic over $\Zb$ is the following $11_3$-configuration:
\begin{center}
$\begin{array}{ccccccccccc} L_1 & L_2& L_3 & L_4 & L_5 & L_6 & L_7 & L_8 & L_9 & L_{10} & L_{11} \\1&1&1&2&2&2&3&3&3&4&4\\6&7&8&5&6&7&4&5&7&5&6\\11&10&9&11&10&9&11&10&8&9&8\end{array}$
\end{center}

A magic labeling of this configuration is provided in Figure~\ref{fig:smallest_Z}.  Note that this labeling is actually \emph{supermagic} as it uses the consecutive integers 1, $\dots$, 11. 
In an upcoming paper, the authors plan to explore magic $n_3$-configurations further -- both supermagic labelings and magic labelings over other abelian groups.

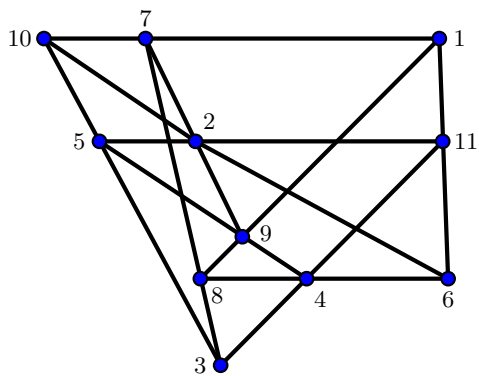
\begin{figure}[ht!]
\captionsetup[subfigure]{labelformat=simple}
\centering
\begin{tikzpicture}[scale=0.9]
\draw[ultra thick] (-1.84,3.56) -- (0.77, -1.27);
\draw[ultra thick] (-1.84, 3.56) -- (4, 3.56);
\draw[ultra thick] (-1.84, 3.56) -- (0.4, 2.04);
\draw[ultra thick] (4.13,0.01) -- (0.4, 2.04);
\draw[ultra thick] (0.47, 0.01) -- (4.13, 0.01);
\draw[ultra thick] (0.47, 0.01) -- (4, 3.56);
\draw[ultra thick] (-0.34, 3.56) -- (1.09, 0.63);
\draw[ultra thick] (-1.02, 2.04) -- (4.05, 2.04);
\draw[ultra thick] (0.77, -1.27) -- (4.05, 2.04);
\draw[ultra thick] (-1.02, 2.04) -- (2.04, 0.01);
\draw[ultra thick] (4, 3.56) -- (4.13, 0.01);
\draw[ultra thick] (-0.34, 3.56) -- (0.77, -1.27);
\draw[thick,fill=blue] (-1.84, 3.56) circle (0.1);
\draw[thick,fill=blue] (4, 3.56) circle (0.1);
\draw[thick,fill=blue] (-1.02, 2.04) circle (0.1);
\draw[thick,fill=blue] (4.05, 2.04) circle (0.1);
\draw[thick,fill=blue] (4.13, 0.01) circle (0.1);
\draw[thick,fill=blue] (0.47, 0.01) circle (0.1);
\draw[thick,fill=blue] (0.77, -1.27) circle (0.1);
\draw[thick,fill=blue] (-0.34, 3.56) circle (0.1);
\draw[thick,fill=blue] (2.04, 0.01) circle (0.1);
\draw[thick,fill=blue] (1.09, 0.63) circle (0.1);
\draw[thick,fill=blue] (0.4, 2.04) circle (0.1);
\draw(-2.2,3.56) node {\small $10$};
\draw(-0.34,3.86) node {\small $7$};
\draw(4.3,3.56) node {\small $1$};
\draw(4.4,2.04) node {\small $11$};
\draw(4.13,-0.29) node {\small $6$};
\draw(0.72,-0.24) node {\small $8$};
\draw(0.47,-1.27) node {\small $3$};
\draw(-1.32,2.04) node {\small $5$};
\draw(2.24,-0.29) node {\small $4$};
\draw(0.6,2.34) node {\small $2$};
\draw(1.44,0.68) node {\small $9$};
\end{tikzpicture}
\caption{The smallest configuration which is magic over $\Zb$.}
\label{fig:smallest_Z}
\end{figure}

\end{document}